\documentclass[12pt]{article}

\usepackage{amsfonts,amssymb,amsmath,amsthm,epsfig,euscript}

\usepackage{tikz}
\usetikzlibrary{matrix,trees,arrows}
\usepackage{graphics,graphicx}                 
\usepackage{xypic}

\usetikzlibrary{positioning}
\usetikzlibrary{fit}
\usetikzlibrary{patterns}

\DeclareMathOperator{\rmax}{rmax}

\newtheorem{theorem}{Theorem}
\newtheorem{lemma}[theorem]{Lemma}
\newtheorem{corollary}[theorem]{Corollary}
\newtheorem{proposition}[theorem]{Proposition}
\newtheorem{remark}[theorem]{Remark}

\theoremstyle{definition}
{

}

\long\def\symbolfootnote[#1]#2{\begingroup
\def\thefootnote{\fnsymbol{footnote}}\footnote[#1]{#2}\endgroup}

\title{Harmonic numbers, Catalan's triangle and mesh patterns}

\author{
Sergey Kitaev \\
\small Department of Computer and Information Sciences\\[-0.8ex]
\small University of Strathclyde\\[-0.8ex]
\small Glasgow G1 1XH, United Kingdom\\[-0.8ex]
\small \texttt{sergey.kitaev@cis.strath.ac.uk}
\and
Jeffrey Liese \\
\small Department of Mathematics\\[-0.8ex]
\small Cal Poly\\[-0.8ex]
\small San Luis Obispo, CA 93407, USA\\[-0.8ex]
\small \texttt{jliese@calpoly.edu}
}

\date{\small Submitted: Date 1;  Accepted: Date 2;
 Published: Date 3.\\
\small MR Subject Classifications: 05A15}

\begin{document}
\maketitle

\begin{abstract}
\noindent \

The notion of a mesh pattern was introduced recently, but it has already proved to be a useful tool for description purposes related to sets of permutations.  In this paper we study eight mesh patterns of small lengths. In particular, we link avoidance of one of the patterns to the harmonic numbers, while for three other patterns we show their distributions on 132-avoiding permutations are given by the Catalan triangle. Also, we show that two specific mesh patterns are Wilf-equivalent.  As a byproduct of our studies, we define a new set of sequences counted by the Catalan numbers and provide a relation on the Catalan triangle that seems to be new.\\


\end{abstract}

\section{Introduction}

The notion of mesh patterns in permutations was introduced by Br\"and\'en and Claesson \cite{BrCl} to provide explicit expansions for certain permutation statistics as possibly infinite linear combinations of (classical) permutation patterns (see \cite{kit} for a comprehensive introduction to the theory of permutation patterns).  There is a line of papers \cite{AKV,HilJonSigVid,kitrem,kitrem2,kitrem3,kitremtie1,kitremtie2,Ulf,Ulf1} related to studying various mesh patterns in sets of permutations or sometimes in restricted sets of permutations, and this paper is a contribution to the study.  In particular, we provide links to the harmonic numbers and the Catalan triangle.  Besides being interesting in their own right, there are other motivations to analyze mesh patterns.  For example, a certain generalization of the notion of mesh patterns was used in \cite{Ulf} by \'Ulfarsson to simplify a description of Gorenstein Schubert varieties and to give a new description of  Schubert varieties that are defined by inclusions.

We will now provide some definitions that will be used throughout the paper.  We define an \textit{$n$-permutation} to be a word without repeated elements over the set $\{1,2,\ldots,n\}$.  An element $\pi_i$ of a permutation $\pi_1\pi_2\cdots\pi_n$ is a {\em right-to-left maximum} if $\pi_i>\pi_j$ for $j\in\{i+1,i+2,\ldots,n\}$. For example, the set of right-to-left maxima of the permutation 264513 is $\{3,5,6\}$.

A mesh pattern is a generalization of several classes of patterns studied intensively in the literature during the last decade (see \cite{kit}). However for this paper we do not need the full definition of classical pattern avoidance.  In fact, apart from the notion of a mesh pattern, we only need the notion of a permutations {\em avoiding the (classical) pattern $132$}, or a {\em $132$-avoiding permutation}. A permutation $\pi= \pi_1\pi_2\cdots\pi_n$ {\em avoids} the pattern 132, if there are no numbers $1\leq i<j<k\leq n$ such that $\pi_i<\pi_k<\pi_j$. For example, the permutation $43512$ avoids the pattern 132 while $24531$ contains two occurrences of this pattern, namely the subsequences 243 and 253. For two patterns (of any type) $p$ and $q$ we say that $p$ and $q$ are {\em Wilf-equivalent} if for all $n\geq 0$, the number of $n$-permutations avoiding $p$ is equal to that avoiding $q$.

The notion of a mesh pattern can be best described using permutation diagrams, which are  similar to permutation matrices (for a more detailed description, we refer to \cite{BrCl,Ulf}). For example, the diagrams in Figure \ref{three-mesh-patterns}, after ignoring the shaded areas and paying attention to the height of the dots (points) while going through them from left to right, each correspond to the permutation 213, while the diagram on the left in Figure \ref{example-permutation} corresponds to the permutation 82536174. A mesh pattern consists of the diagram corresponding to a permutation where some subset of the squares determined by the grid are shaded.  In fact, three mesh patterns are depicted in Figure \ref{three-mesh-patterns} and eight mesh patterns in Figure \ref{allPatterns}.

\begin{figure}[ht]
\begin{center}
\includegraphics[scale=0.4]{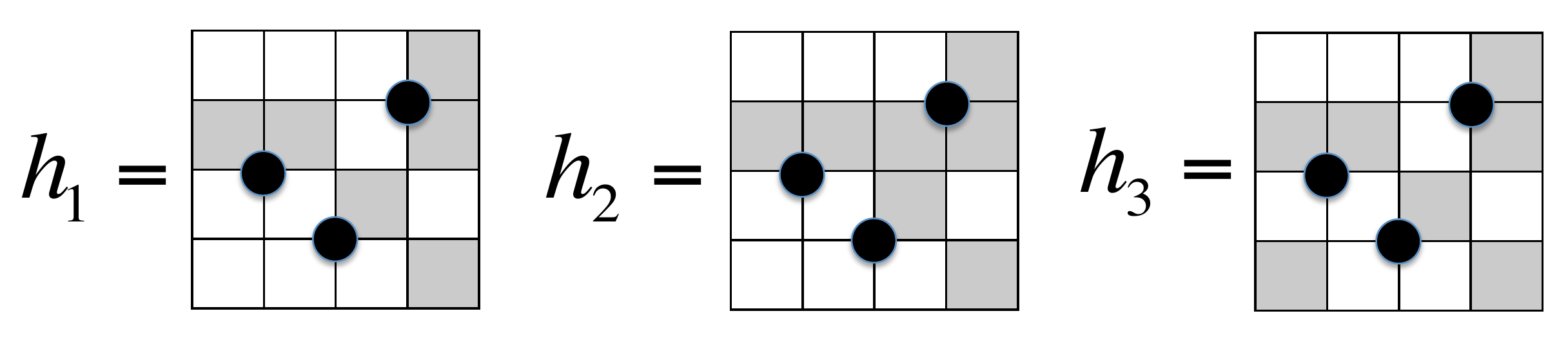}
\end{center}
\vspace{-15pt}
\caption{Three mesh patterns.}\label{three-mesh-patterns}
\end{figure}

We say that a mesh pattern $p$ of length $k$ occurs in a permutation $\pi$ if the permutation diagram of $\pi$ contains $k$ dots whose order is the same as that of the permutation diagram of $p$, i.e. $\pi$ contains a subsequence that is order-isomorphic to $p$, and additionally, no element of $\pi$ can be present in a shaded area determined by $p$ and the corresponding elements of $\pi$ in this subsequence.  For example, the three circled elements in the permutation $82536174$ in Figure \ref{example-permutation} are an occurrence of the mesh pattern $h_1$ defined in Figure \ref{three-mesh-patterns}, as demonstrated by the diagram on the right in Figure \ref{example-permutation}  (note that none of the permutation elements fall into the shaded area determined by the mesh pattern).  However, these circled elements are not an occurrence of the pattern $h_2$ in Figure \ref{three-mesh-patterns} because of the element 6 in the permutation; they also are not an occurrence of the pattern $h_3$ in Figure \ref{three-mesh-patterns} because of the element 2 in the permutation. One can verify using the diagram in Figure \ref{three-mesh-patterns} that the subsequence 536 in the permutation $82536174$ is an occurrence of the mesh patterns $h_1$ and $h_2$, but not $h_3$.

\begin{figure}[ht]
\begin{center}
\includegraphics[scale=0.4]{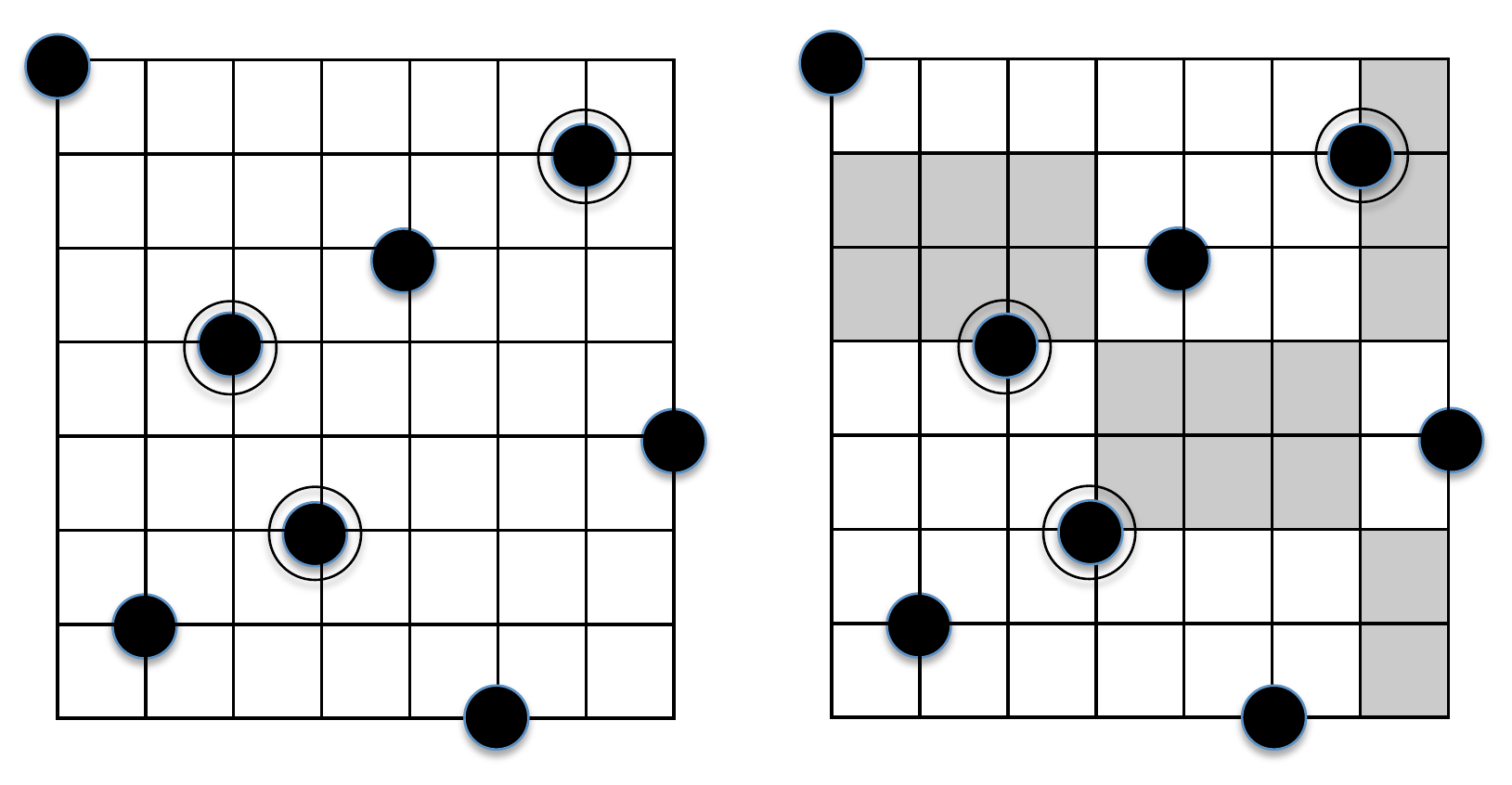}
\end{center}
\vspace{-15pt}
\caption{An example of an occurrence of a mesh pattern.}\label{example-permutation}
\end{figure}

The definitions of all mesh patterns of interest in this paper are given in Figure \ref{allPatterns}. In particular, $p$ is an instance of what we call {\em border mesh patterns}, which are defined by shading all non-interior squares.

\begin{figure}[ht]
\begin{center}
\includegraphics[scale=0.4]{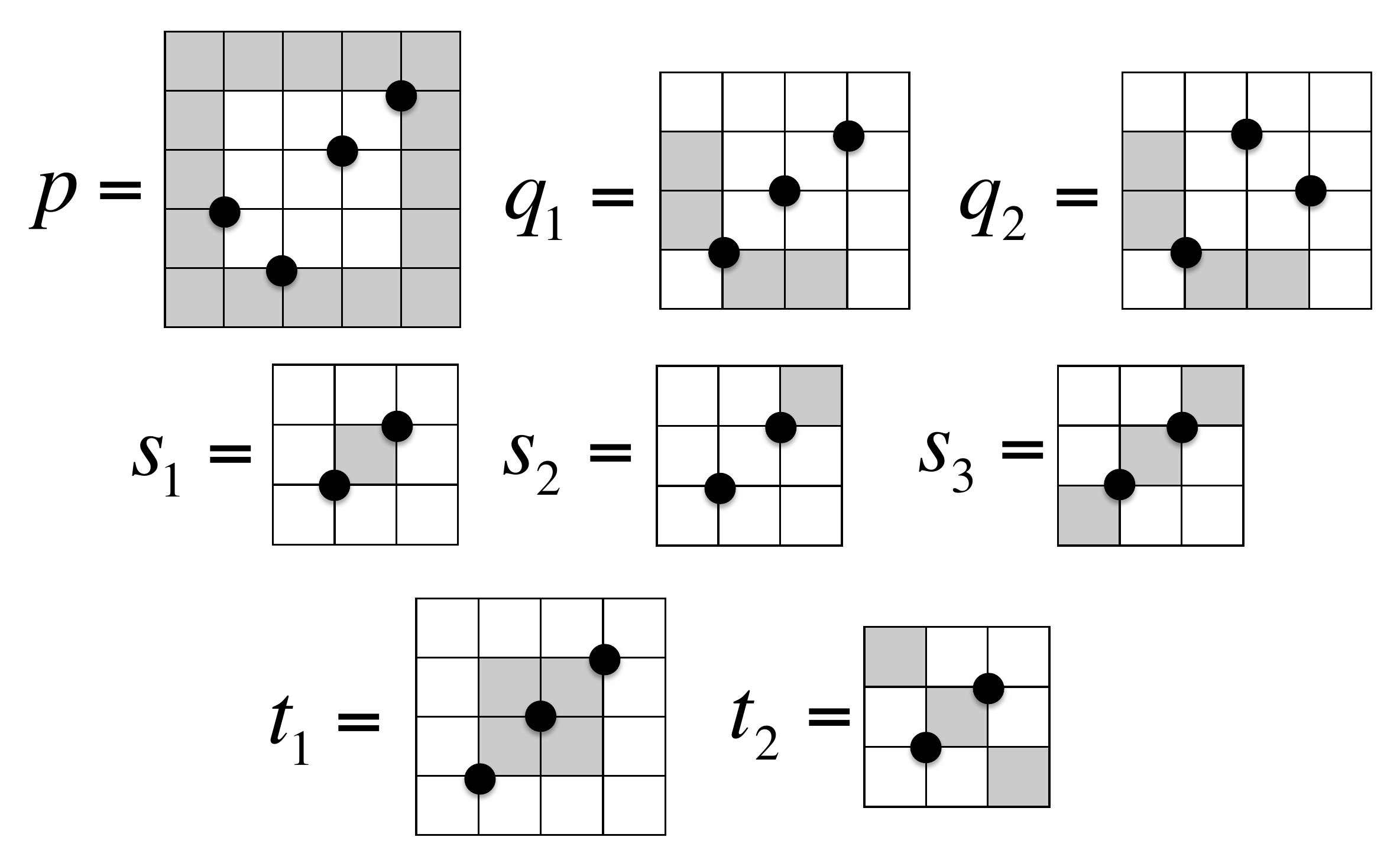}
\end{center}
\vspace{-15pt}
\caption{Definitions for all patterns of interest in this paper.}\label{allPatterns}
\end{figure}

We also need to define a Catalan number and the Catalan triangle. The \textit{$n$-th Catalan number} $C_n$ is defined by the recursion $C_{n+1}=\sum_{i=0}^{n}C_iC_{n-i}$ with $C_0=1$.  \textit{Catalan's triangle} is defined by $C(0,0)=1$, $C(0,k)=0$ for $k>0$ and $C(n,k)=C(n-1,k)+C(n,k-1)$.  An alternative recursion for the Catalan triangle, namely $C(n,k)=\sum_{j=0}^{k}C(n-1,j)$, will also be useful.  The beginning of Catalan's triangle is shown below.
$$
\begin{array}{ccccccccccc}
& & & & & 1 & & & & & \\
& & & &  1 &  & 1 & & & & \\
& & & 1 & & 2 & & 2 & & & \\
& & 1 & & 3 & & 5 & & 5 & & \\
& 1 & & 4 & & 9 & & 14 & & 14 & \\
1 & & 5 & & 14 & & 28 & & 42 & & 42
\end{array}
$$
The Catalan numbers can always be read from Catalan's triangle by looking at the rightmost number in each row.  The following two formulas for the Catalan numbers and the entries in Catalan's triangle are both well-known:
$$C_n=\frac{1}{n+1}{2n \choose n} \ \ \ \ \ \ C(n,k)=\frac{(n+k)!(n-k+1)}{k!(n+1)!}.$$
If the number of $n$-permutations with $k$ occurrences of a pattern $\tau$ is given by $C(n,k)$ or a shift of indices of this number, e.g. $C(n-1,k)$, we say that $\tau$ has {\em Catalan's distribution}.  \\

The paper is organized as follows. In Section \ref{harmonic} we not only link the distribution of the pattern $p$ to the harmonic numbers (see Theorems \ref{lem:pnkSatisfiesRec} and \ref{avoid2134}), but also study an exponential generating function for this distribution (see Theorem \ref{thmGenerFunc1}). In Section \ref{sec-q1-q2} we show Wilf-equivalence of the patterns $q_1$ and $q_2$. Section \ref{sec-132-av} is devoted to study of the patterns $s_1$, $s_2$, $s_3$, $t_1$ and $t_2$ on $132$-avoiding permutations. More specifically, Subsections \ref{sub-s1} and \ref{sub-s2} show that both $s_1$ and $s_2$ have Catalan's distribution on $132$-avoiding permutations (see Theorems \ref{thm:Boxed12IsCatalanTriangle} and \ref{thm:12TopRightIsCatalanTriangle}), while Subsection \ref{sub-s3} shows that $s_3$ has the reverse Catalan distribution on this class of permutations (see Theorem \ref{thm:ReverseDiagonal12IsReverseCatalanTriangle}). We prove Theorem \ref{thm:ReverseDiagonal12IsReverseCatalanTriangle} combinatorially by introducing an involution on the set of 132-avoiding permutations and using Theorem \ref{thm:Boxed12IsCatalanTriangle}. This involution allows us to establish a joint equidistribution fact for four statistics on 132-avoiding permutations. We also provide an extra proof of Theorem \ref{thm:Boxed12IsCatalanTriangle} to establish a bijective proof of the fact that $s_1$ and $s_2$ are equidistributed. As a byproduct to our research, we define a new set of sequences counted by the Catalan numbers (see Proposition \ref{byproductCatalan1}). Additionally, we discover a relation for Catalan's triangle that involves the Catalan numbers which seems to be new (see Theorem \ref{thm:CatalanTriangleRelation}).  This relation led us to a combinatorial proof of a binomial identity stated in Corollary \ref{binom}. In Subsection \ref{sub-t1} we discuss the minimum and maximum number of occurrences of the pattern $t_1$ on 132-avoiding permutations (see beginning of the subsection and Theorem \ref{maxt1}), while in Subsection \ref{sub-t2} we find the number of 132-avoiding permutations with exactly zero, one, two or three occurrences of the pattern $t_2$ (see Theorems \ref{q132proposition1}, \ref{q132proposition2}, \ref{q132proposition3}, \ref{q132proposition4}); essentially all our results here are given in terms of the Catalan numbers. Finally, in Section \ref{concluding} we provide some concluding remarks.

\section{The pattern $p$ and the harmonic numbers}\label{harmonic}

We let $H_n=\sum_{k=1}^{n}\frac{1}{k}$ denote the $n$-th harmonic number. In this section, we will express the distribution of the border mesh pattern $p$ in terms of $H_n$.

\begin{proposition} For $n\geq 4$ and $k\geq1$, we have that \begin{equation}\label{for:pnkFormula}p_{n,k}:=(n-2)!\sum_{i=k+1}^{n-2}\frac{1}{i}=(n-2)!(H_{n-2}-H_k)\end{equation} satisfies the recursion \begin{equation}\label{eqn:pnkRecursion}p_{n,k}=(n-2)p_{n-1,k}+(n-3)!.\end{equation}\end{proposition}

\begin{proof} Plugging in the formula into the RHS of (\ref{eqn:pnkRecursion}) we obtain,
\begin{eqnarray*}
(n-2)p_{n-1,k}+(n-3)!&=&(n-2)\left((n-3)!\sum_{i=k+1}^{n-3}\frac{1}{i}\right)+(n-3)!\\
&=&(n-2)!\sum_{i=k+1}^{n-3}\frac{1}{i}+(n-3)!\\
&=&(n-2)!\sum_{i=k+1}^{n-2}\frac{1}{i}\\
&=&p_{n,k}.
\end{eqnarray*}
We are done.\end{proof}

\begin{theorem} \label{lem:pnkSatisfiesRec} The number of $n$-permutations with $k$ occurrences of the border mesh pattern $p$ for $k\geq1$ is given by $p_{n,k}$.  \end{theorem}

\begin{proof}
We prove this theorem by showing that the number of $n$-permutations with $k$ occurrences of the border mesh pattern $p$ also satisfies {\rm(\ref{eqn:pnkRecursion})} and because of the matching initial conditions, the result follows.

To show that the number of such permutations satisfies {\rm(\ref{eqn:pnkRecursion})}, it helps first to make a few observations about the pattern $p$.  Firstly, given any occurrence of $p$ in an $n$-permutation, $1$ and $n$ must play the roles of $1$ and $4$ respectively in the pattern.  Additionally, $n$ must be the last element of the permutation and the first element of the permutation must play the role of 2 in the pattern.  It is then easy to deduce that the number of elements between $1$ and $n$ in such a permutation is precisely the number of occurrences of $p$.

We will now count permutations having exactly $k$ occurrences of $p$.  Either the permutation begins with a 2 or it does not.  Suppose the permutation does begin with a 2.  As mentioned before it must also end with $n$ and, by the above observations, in order for it to have exactly $k$ occurrences of the pattern, $1$ must be fixed in position $n-k-1$.  We are then free to permute the remaining elements in $(n-3)!$ ways.  Thus, the number of permutations beginning with a 2 and having $k$ occurrences of the pattern is $(n-3)!$.

Now suppose the permutation begins with a number other than 2.  Notice that removing the element $2$ from the permutation and relabeling yields an $(n-1)$-permutation still having $k$ occurrences of the pattern.  Also, $2$ could have been in any position other than the first and the last so in total there are $(n-2)$ possible positions for 2.  Thus, the number of permutations beginning with an element other than two and having $k$ occurrences of the pattern is $(n-2)p_{n-1,k}$ and the theorem is proved.
\end{proof}

It turns out that $p_{n,0}$, the number of $n$-permutations avoiding $p$, does not satisfy (\ref{eqn:pnkRecursion}). However, from (\ref{for:pnkFormula}), it is clear that we get the following relation $p_{n,k}=p_{n,k-1}-\frac{(n-2)!}{k}$, which does mean that we can always get $p_{n,k}$ in terms of $p_{n,1}$.  Using this fact, one can get a formula for $p_{n,0}$, the number of $n$-permutations which avoid $p$.

\begin{theorem}\label{avoid2134}  We have $$p_{n,0}= n!-(n-3)(n-2)!+p_{n,1}=(n-2)!(H_{n-2}+n^2-2n+2).$$ \end{theorem}

\begin{proof}
For any $n$-permutation, the maximum number of occurrences of the pattern $p$ is $n-3$, so to compute the number of $n$-permutations avoiding the pattern we only need to subtract the number of permutations having $k$ occurrences for $k$ from $1$ to $n-3$ from $n!$.  This gives
\begin{eqnarray*}
p_{n,0}&=&n!-\sum_{j=1}^{n-3}p_{n,j}\\
&=&n!-(n-2)!\sum_{j=1}^{n-3}\sum_{i=j+1}^{n-2}\frac{1}{i}\\
&=&n!-(n-2)!\sum_{i=2}^{n-2}\frac{1}{i}(i-1)\\
&=&n!-(n-2)!\left((n-3)-\sum_{i=2}^{n-2}\frac{1}{i}\right)\\
&=&n!-(n-3)(n-2)!+p_{n,1}.\\
\end{eqnarray*}
We are done.\end{proof}

Next we study exponential generating functions (e.g.f.s) for the numbers $p_{n,k}$.

 \begin{theorem}\label{thmGenerFunc1} Let $$P_k(t)=\sum_{n\geq k+3}\frac{p_{n,k}}{(n-2)!}t^{n-2}.$$ Then, $P_k(t)$ satisfies the following differential equation with initial condition $P_k(0)=0$: $$P'_k(t)=\frac{k!P_k(t)+t^k}{k!(1-t)}+\frac{t^{k+1}}{(1-t)^2}.$$\end{theorem}

\begin{proof} Note that $p_{k+3,k}=k!$ and $p_{i,k}=0$ for $i\leq k+2$. We begin with the recursion proven in Theorem~\ref{lem:pnkSatisfiesRec}.  Namely,
$$p_{n,k}=(n-2)p_{n-1,k}+(n-3)!,$$ which we can write as
$$p_{n,k}=(n-3)p_{n-1,k}+p_{n-1,k}+(n-3)!.$$
Multiply both sides by $\frac{t^{n-3}}{(n-3)!}$ and sum over all $n\geq k+4$:
$$\sum_{n\geq k+4}\frac{p_{n,k}t^{n-3}}{(n-3)!}=t\sum_{n\geq k+4}\frac{p_{n-1,k}t^{n-4}}{(n-4)!}+\sum_{n\geq k+4}\frac{p_{n-1,k}t^{n-3}}{(n-3)!}+\sum_{n\geq k+4}t^{n-3}$$
$$P'_k(t)-\frac{t^k}{k!}=tP'_k(t)+P_k(t)+\frac{t^{k+1}}{1-t}.$$
And thus,
$$P'_k(t)=\frac{k!P_k(t)+t^k}{k!(1-t)}+\frac{t^{k+1}}{(1-t)^2}.$$
\end{proof}

\begin{corollary} Solving the differential equations in Theorem {\rm \ref{thmGenerFunc1}} for $k=1,2,3,4$ we get e.g.f.s that we record in Table {\rm\ref{k1234}}. Solving this equation for general $k$ using Mathematica  produces the answer:
$$P_k(t)=\frac{{t^{1 + k}} _2F_1[1, 1 + k, 2 + k, t]}{(1 + k) (1 - t)},$$ where the hypergeometric function is defined for $|t|<1$ by the power series
$$_2F_1(a,b;c;t)=\sum_{n=0}^{\infty}\frac{(a)_n(b)_n}{(c)_n}\frac{t^n}{n!}$$ provided that $c$ does not equal $0$, $-1$, $-2,\ldots$. Here $(q)_n$ is the Pochhammer symbol defined by
$$(q)_n=\begin{cases} 1 & \mbox{if }n=0, \\ q(q+1)\cdots(q+n-1) & \mbox{if }n>0. \end{cases}$$
\end{corollary}

\begin{table}
\begin{center}
\begin{tabular}{c|c|c}
$k$ & e.g.f. & expansion\\
\hline
& \\
1  & $\frac{t+\ln(1-t)}{t-1}$ & $1\frac{t^2}{2!}+5\frac{t^3}{3!}+26\frac{t^4}{4!}+154\frac{t^5}{5!}+1044\frac{t^6}{6!}+\ldots$\\
& \\
2  & $\frac{2 t+t^2+2 \ln(1-t)}{2 (t-1)}$ & $2\frac{t^3}{3!}+14\frac{t^4}{4!}+94\frac{t^5}{5!}+684\frac{t^6}{6!}+5508\frac{t^7}{7!}+\ldots$\\
& \\
3  & $\frac{6 t+3 t^2+2 t^3+6 \ln(1-t)}{6 (t-1)}$ & $6\frac{t^4}{4!}+54\frac{t^5}{5!}+444\frac{t^6}{6!}+3828\frac{t^7}{7!}+35664\frac{t^8}{8!}+\ldots$\\
& \\
4  & $\frac{12 t+6 t^2+4 t^3+3 t^4+12 \ln (1-t)}{12 (t-1)}$ & $24\frac{t^5}{5!}+264\frac{t^6}{6!}+2568\frac{t^7}{7!}+25584\frac{t^8}{8!}+270576\frac{t^9}{9!}+\ldots$\\
& \\
\end{tabular}
\caption{The e.g.f.s defined in Theorem {\rm \ref{thmGenerFunc1}} for the number of permutations with $k=1,2,3,4$ occurrences of the pattern $p$.}\label{k1234}
\end{center}
\end{table}

\section{Wilf-equivalence of $q_1$ and $q_2$}\label{sec-q1-q2}

In this section we prove the Wilf-equivalence of the patterns $q_1$ and $q_2$ defined in Figure \ref{allPatterns}.

Suppose that $1\leq x_1<x_2<\cdots<x_k\leq n$ and $1\leq y_1<y_2<\cdots<y_k\leq n$ are respective positions of two occurrences of a length $k$ pattern in an $n$-permutation. Then we say that the first occurrence is to the left of the second occurrence if $(x_1,x_2,\ldots,x_k)$ is lexicographically smaller than $(y_1,y_2,\ldots,y_k)$. Clearly, this relation defines a total order on the set of all occurrences of the pattern in the permutation.  Thus, unless a permutation happens to avoid a pattern, there must be the smallest occurrence, namely the leftmost one.

We establish Wilf-equivalence of $q_1$ and $q_2$ by providing a well-defined bijective map $g$ that turns $q_1$-avoiding permutations into $q_2$-avoiding permutations.

Given a $q_1$-avoiding permutation $\pi$ that also avoids $q_2$, we let $g(\pi)=\pi$, which is trivially bijective and well-defined on the set of $(q_1,q_2)$-avoiding permutations.

On the other hand, if a $q_1$-avoiding permutation $\pi$ contains an occurrence of $q_2$, we apply the following procedure.  Take the leftmost (lexicographically smallest) occurrence of $q_2$ and swap the second and third elements of this particular occurrence as shown schematically in Figure \ref{132versus123} below.  It is easy to see that an occurrence of $q_2$ will be turned into an occurrence of $q_1$.  Now repeat the procedure until there are no longer any occurrences of $q_2$ in the permutation to obtain $g(\pi)$. \\

\noindent
{\bf $g$ is well-defined.} First, we must justify that  the process described above terminates, thus resulting in a $q_2$-avoiding permutation.  To this end, suppose that $xyz$ is the leftmost occurrence of $q_2$ in $\pi$ (occurring in the positions $a<b<c$) as shown in the left half of Figure  \ref{132versus123}.  We will now verify that exchanging $y$ and $z$ does not create an occurrence of $q_2$ which is lexicographically smaller than the smallest occurrence of $q_2$ in $\pi$.  Thus at any point in the process, turning the leftmost occurrence of $q_2$ into an occurrence of $q_1$ in the prescribed way will never introduce an occurrence of $q_2$ which is lexicographically smaller than the occurrence we have just removed.  This ensures that the process will eventually terminate.  In the following discussion we refer to Figure \ref{132versus123}.

\begin{figure}[h]
\begin{center}
\includegraphics[scale=0.5]{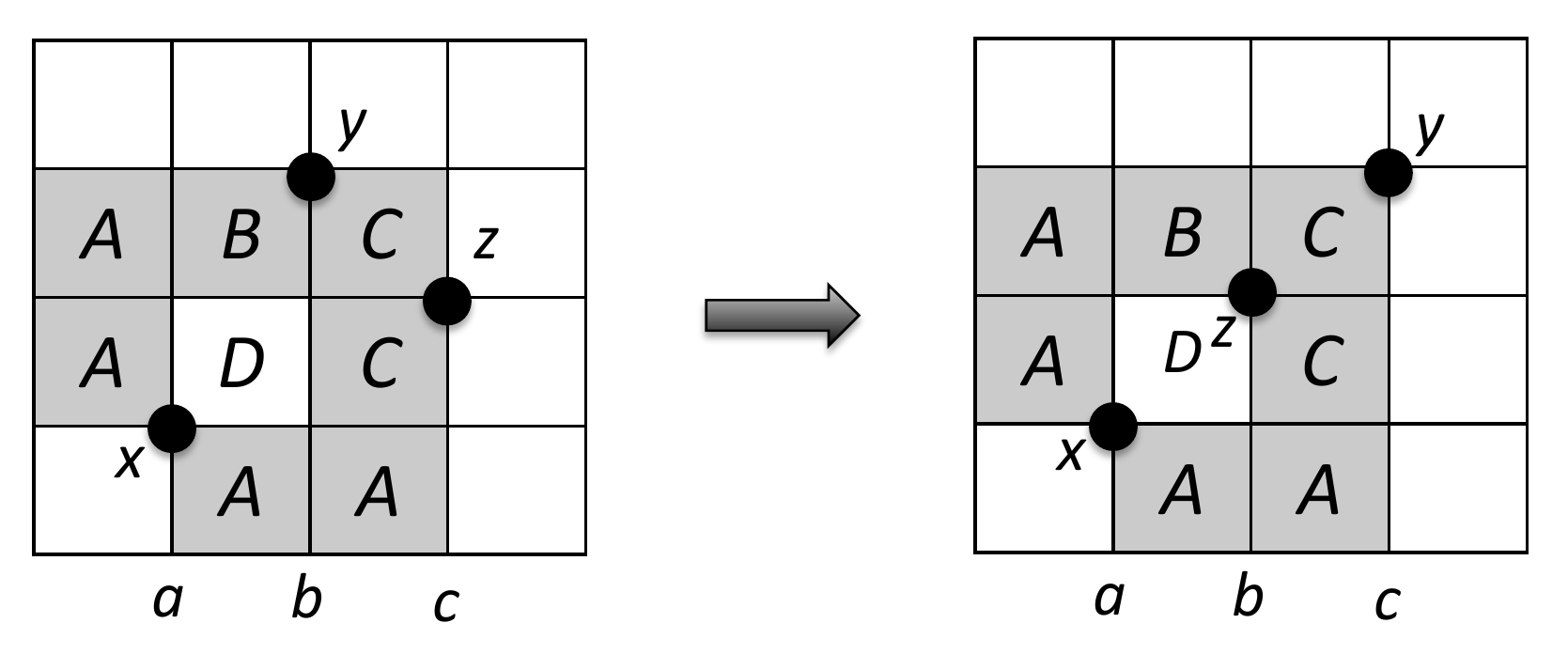}
\vspace{-10pt}
\caption{Turning an occurrence of $q_2$ into an occurrence of $q_1$.}\label{132versus123}
\end{center}
\end{figure}

Indeed, the areas marked by $A$ in $\pi$ must be empty because $xyz$ is an occurrence of $q_2$ (we indicate this by shading the area). The area marked by $B$ must  be empty because $xyz$ is the lexicographically smallest occurrence (otherwise any element $y'\in B$ would give an occurrence of $q_2$, $xy'z$, which is lexicographically smaller than $xyz$). By a similar reason the areas marked by $C$ must be empty (otherwise, any element $z'\in C$ would give an occurrence of $q_2$, $xyz'$, which is lexicographically smaller than $xyz$). Also note that at the first step of the algorithm, the area marked by $D$ must be empty, because otherwise any element $t\in D$ would give an occurrence of $q_1$ in $\pi$, namely $xty$, but $\pi$ is $q_1$-avoiding.  However, after some number of occurrences of $q_2$ have been turned into occurrences of $q_1$, $D$ is not guaranteed to be empty.  Therefore, we do not shade $D$ as we can make an argument that applies to an arbitrary step of the algorithm described.  Note that we can say something about the elements inside $D$, namely that if $D$ contains two or more elements, these elements must be increasing (otherwise, $xyz$ would not be the smallest occurrence of $q_2$).

Suppose now that turning $xyz$ into $xzy$ creates an occurrence $x'y'z'$ of $q_2$ that is lexicographically smaller than $xyz$.  It must be the case that $x'=x$. To justify this, note that if $x'\neq x$, then one of $y'$ or $z'$ must be $y$ or $z$, because if this weren't so, $x'y'z'$ would have been a preexisting occurrence of $q_2$ that was to the left of $xyz$, a contradiction.  Since one of $y'$ or $z'$ must be $y$ or $z$, $x'$ could not possibly be located in the top left unshaded area.  Thus, $x'<x$, and like before $x'yz$ would have been a preexisting occurrence of $q_2$ to the left of $xyz$, again a contradiction.  This fully justifies that $x=x'$, which then implies that $z'=z$ or $z'=y$.  This is simply because as mentioned before, one of $y'$ or $z'$ must be $y$ or $z$, but $y'\neq y$ and $y'\neq z$ as if $y'$ were either $y$ or $z$ this would force $xy'z'$ to be to the right of $xyz$, another contradiction.  Thus it must be the case that $z'=z$ or $z'=y$.

So, either $xy'z$ or $xy'y$ is an occurrence of $q_2$ to the left of $xyz$. In the first case, $y'$ is above the $B$ area, and thus $xy'y$ would be an occurrence of $q_2$ in $\pi$ which is to the left of $xyz$, a contradiction. In the second case, $y'$ is again above the $B$ area, because if it were above the $C$ areas, $xy'y$ would not be to the left of $xyz$, and we get exactly the same contradiction. Thus we do not create an occurrence of $q_2$ to the left of $xyz$ at any particular step of the algorithm.  Therefore, the process is well-defined and it eventually terminates when all of occurrences of $q_2$ have been removed.

Note that the number of occurrences of $q_2$ in $\pi$ is not necessarily equal to the number of occurrences of $q_1$ in $g(\pi)$. For example, the permutation 1432 has 3 occurrences of $q_2$, while $g(1432)=1234$ has 4 occurrences of $q_1$. \\

\noindent
{\bf $g$ is invertible.} In order to show that $g$ is invertible, and thus is bijective, we need to analyze sequences of consecutive steps of the algorithm which leads us to considering a more refined structure presented in Figure \ref{132versus123refined}.  Like in Figure \ref{132versus123}, suppose that $x$ is the smallest element of the leftmost occurrence of $q_2$.

\begin{figure}[h]
\begin{center}
\includegraphics[scale=0.5]{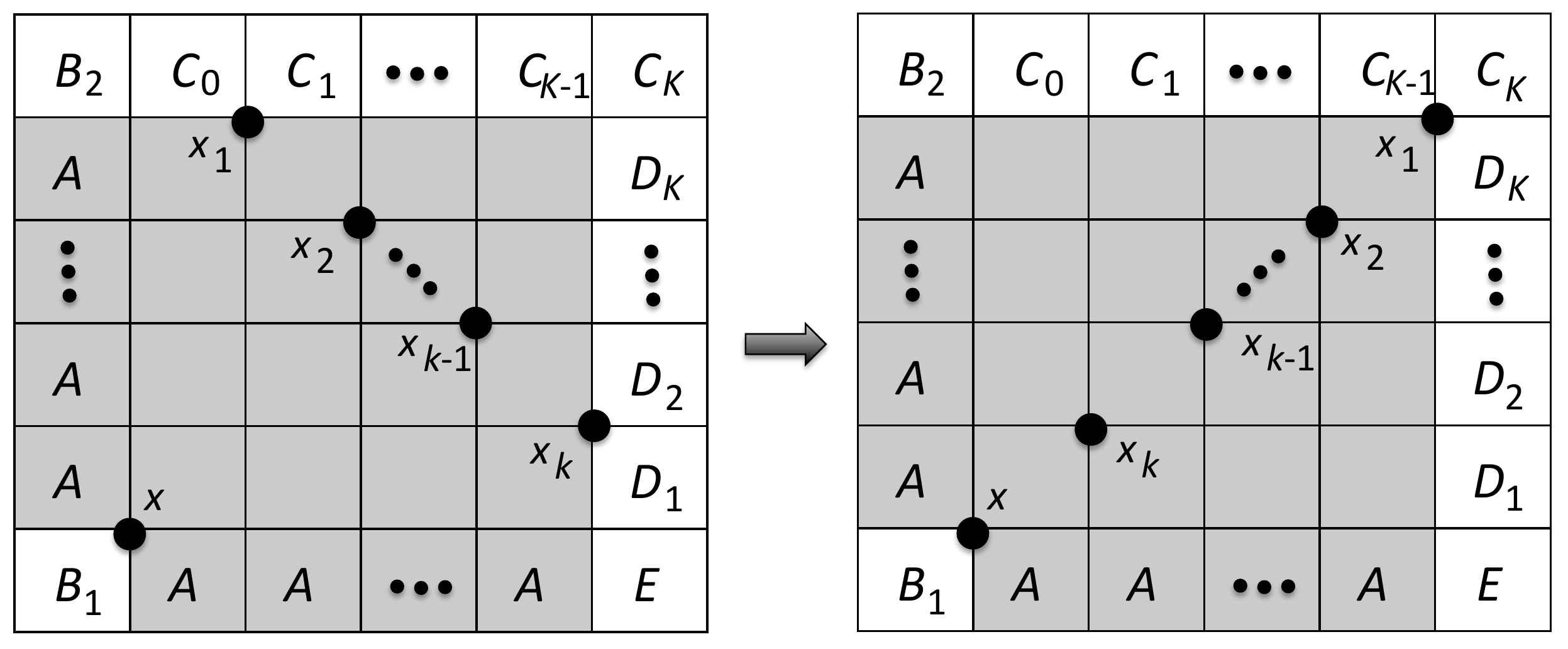}
\vspace{-10pt}
\caption{A sequence of steps in turning a cluster of occurrences of $q_2$ into occurrences of $q_1$.}\label{132versus123refined}
\end{center}
\end{figure}

Define $M_x=\{x,x_1,x_2,\ldots,x_k\}$ be the set of all elements involved in any occurrence of $q_2$ having $x$ as the smallest element.  We shall refer to the set $M_x$ of elements as a \textit{cluster} of occurrences of $q_2$ and it can be readily seen that this cluster contains $\binom{k}{2}$ occurrences of $q_2$.  By the definition of $M_x$, the squares marked with $A$ in Figure \ref{132versus123refined} must be empty.  Also, note that the elements $x_1$, $x_2,\ldots,x_k$ must be in decreasing order from left to right as shown in Figure \ref{132versus123refined}.  This is because if $x_i>x_j$ for $i<j$ then $xx_ix_j$ would be an occurrence of $q_1$, which is a contradiction on the first step because the permutation is $q_1$-avoiding.  It is also a contradiction at any other step because we will soon show that when performing the steps of the algorithm we never introduce an occurrence of $q_1$ that doesn't have $x$ as the smallest element, i.e. we never introduce $q_1$ into any other cluster.  By a similar argument, the elements inside the square defined by the elements $x$, $x_1$ and $x_k$ all belong to the set $M_x$.

We will now make three important observations.  First, the elements in $M_x$ cannot be involved in occurrences of $q_2$ involving elements outside of $M_x$.  For $x$ it follows by definition of $M_x$, while if $x_i$ is the smallest element of an occurrence $x_iab$ of $q_2$ then $xab$ would also be an occurrence of $q_2$ and thus $a,b$ must belong to $M_x$.  If $x_i$ is the largest or second largest element of an occurrence $ax_ib$ or $abx_i$, then $a$ must lie inside the square labeled $B_1$, contradicting the fact that $x$ is involved as the smallest element in the leftmost occurrence of $q_2$.  In short, clusters are disjoint.

Second, note that any other cluster of occurrences of $q_2$ must be entirely contained inside a single square labeled with $C_i$ or $D_j$.  This is true because if there were an occurrence of $q_2$ that involved two of these squares, one can show that at least one of the elements in $M_x$ would be present in a shaded area for the pattern $q_2$.  Thus when performing the steps of the algorithm, we are guaranteed to never introduce an occurrence of $q_1$ whose smallest element also happens to be the smallest element of some other cluster.

After performing $\binom{k}{2}$ steps of our algorithm to the cluster of occurrences of $q_2$ pictured on the left in Figure \ref{132versus123refined}, we find ourselves in the situation pictured on the right in that figure.  We now make our third observation, namely that the elements in $M_x$ now cannot be involved in any occurrences of $q_1$ involving elements outside of $M_x$.  To justify this, note that any element in the square labeled $B_1$ could not have been the smallest element of an occurrence of $q_2$ before the switch, therefore such elements can not be the smallest element of an occurrence of $q_1$ after the switch. Therefore the elements in $M_x$ could only possibly be the smallest elements of occurrences of $q_1$ outside of this cluster.  This is also not possible.  Suppose that some element $z$ of $M_x$ exists such that $zab$ is an occurrence of $q_1$ with $b\not \in M_x$ (certainly if $b\in M_x$, then $a\in M_x$).  If $b$ lies in any square labeled with $D_i$, then $xx_1z$ was an occurrence of $q_2$ before the switch, a contradiction.  If $b$ were in the square labeled with $C_i$ for $1\leq i\leq k$, then $xx_1b$ was an occurrence of $q_1$ before the switch which is also a contradiction by our first observation; a similar argument holds if $b\in C_0$.

All of these observations together lead to the fact that after the algorithm will be implemented, the resulting permutation will contain at least one maximal increasing subsequence consisting of at least three elements so that choosing any three of them yields a $q_1$ pattern, like the sequence $(x,x_k,x_{k-1},\ldots,x_1)$ shown on the right in Figure \ref{132versus123refined}.  Such a subsequence had to be introduced by performing the algorithm to a cluster shown on the left in Figure \ref{132versus123refined}.  It is then straightforward to invert $g$ by turning subsequences of the form $x<x_k<x_{k-1}<\cdots<x_2<x_1$ any three elements of which form an occurrence of the mesh pattern $q_1$ into $xx_1x_2\cdots x_k$ where $k\geq 3$ without changing the positions of $x,x_1,x_2,\ldots,x_k$.

\section{Mesh patterns over 132-avoiding permutations}\label{sec-132-av}

In this section, we consider the mesh patterns $s_1$, $s_2$, $s_3$, $t_1$ and $t_2$ defined in Figure \ref{allPatterns}. Throughout the section, we will be using the following fact that is well-known and is easy to see: If a permutation $\pi$ avoids the (classical) pattern 132, then $\pi=\pi_1n\pi_2$ where every element of $\pi_1$, if any, is larger than every element of $\pi_2$, if any, and $\pi_1$ and $\pi_2$ are any 132-avoiding permutations on their respective elements. It is well-known (see, e.g. \cite{kit}) that there are $C_n$ $n$-permutations avoiding 132.

\subsection{The pattern $s_1$ has Catalan's distribution}\label{sub-s1}

For a permutation $\pi$, we let $N_{\tau}(\pi)$ denote the number of occurrences of a pattern $\tau$ in $\pi$.

\begin{lemma}\label{lemmaEasy1}  For a $132$-avoiding $n$-permutation $\pi$, $N_{s_1}(\pi)$ plus the number of right-to-left maxima in $\pi$ is equal to $n$.\end{lemma}

\begin{proof} We can proceed by induction on the length of the permutation. The case $n=1$ is clear. If $n>1$, then $\pi=\pi_1n\pi_2$ where every element of $\pi_1$, if any, is larger than every element of $\pi_2$, if any.  Because of this fact, there cannot be any occurrence of $s_1$ in $\pi$ having one element in $\pi_1$ and the other one in $\pi_2$. We can now apply the induction hypothesis to $\pi_1n$ and $\pi_2$ to get the desired result:
\begin{eqnarray*}
\rmax(\pi)&=&\rmax(\pi_1n)+\rmax(\pi_2)\\
&=&N_{s_1}(\pi_1n)+N_{s_1}(\pi_2)\\
&=&N_{s_1}(\pi),\\
\end{eqnarray*} and thus the lemma holds.\end{proof}

Let $T(n,k)$ denote the number of $132$-avoiding $n$-permutations with $k$ occurrences of the pattern $s_1$. In particular,  clearly, $T(1,0)=1$ and $T(1,k)=0$ for $k>0$.

\begin{theorem}\label{thm:Boxed12IsCatalanTriangle} We have $T(n,k)=C(n-1,k)$, where $C(n,k)$ is the $(n,k)$-th entry of the Catalan triangle defined in the introduction.  Thus, $$T(n,k)=\frac{(n-k){n-1+k\choose n-1}}{n}.$$\end{theorem}

We supply two proofs of the theorem. The second proof is essentially the exact argument used in the proof of Theorem \ref{thm:12TopRightIsCatalanTriangle}; however, the main reason that we state it here is that it allows us to provide a combinatorial explanation for the fact that the patterns $s_1$ and $s_2$ are equidistributed.

\begin{proof} [First proof of Theorem {\rm\ref{thm:Boxed12IsCatalanTriangle}}] We use the following recurrence for the Catalan triangle: \begin{equation}\label{eq:CatalanTriangleRecurrence}
C(n,k)=\displaystyle\sum_{j=0}^{k}C(n-1,j),\end{equation} with $C(0,0)=1$ and $C(0,k)=0$ for $k>0$.  Since we already know that $T(1,0)=1$ and that $T(1,k)=0$ for $k>0$, we need only show that $T(n+1,k)$ satisfies (\ref{eq:CatalanTriangleRecurrence}).

Consider creating a $132$-avoiding $(n+1)$-permutation from a $132$-avoiding $n$-permutation by inserting $n+1$. The only valid positions to insert $n+1$ is either in front of the $n$-permutation, or right after a right-to-left maximum, otherwise an occurrence of the pattern 132 will be introduced. Also, inserting $n+1$ never eliminates an occurrence of $s_1$. Moreover, inserting $n+1$ in front of an $n$-permutation does not introduce an occurrence of $s_1$, while inserting it immediately to the right of the $i$-th right-to-left maximum (counted from left to right) increases the number of occurrences of $s_1$ by $i$.  This is simply because each right-to-left maximum to the left of $n$, together with $n$, will contribute new occurrences of $s_1$.

We are now able to provide a combinatorial proof of $T(n+1,k)=\displaystyle\sum_{j=0}^{k}T(n,j)$. Take all 132-avoiding $n$-permutations having exactly $j$ occurrences of $s_1$, where $0\leq j\leq k$.  Now for each permutation, insert $n+1$ in the unique place to make the total number of occurrences of $s_1$ equal $k$. Lemma \ref{lemmaEasy1} states for a 132-avoiding $n$-permutation the number of occurrences of $s_1$ plus the number of right-to-left maxima is equal to $n$.  This lemma coupled with the fact that $k\leq n$ guarantees that we always can insert $n$ into the proper place in a permutation counted by $T(n,j)$ to obtain a permutation counted by $T(n+1,k)$ and thus the recursion is verified.
\end{proof}

\begin{proof}[Second proof of Theorem {\rm\ref{thm:Boxed12IsCatalanTriangle}}]  We will prove combinatorially that $T(n,k)$ satisfies $T(n,k)=T(n-1,k)+T(n,k-1)$ for $k<n$, a known recursion for the Catalan triangle.

Again, suppose that $\pi=\pi_1n\pi_2$.  It is not difficult to see that every element in $\pi_1$ is the bottom element of exactly one occurrence of $p$; the top element of this occurrence is either $n$ or the next element larger than itself.  The permutations that correspond to $\pi_1$ being empty are counted by the term $T(n-1,k)$.  We will now show that $T(n,k-1)$ is responsible for counting those permutations where $\pi_1$ is not empty.  We accomplish this by providing a general method to take a permutation $\tau=\tau_1n\tau_2$ counted by $T(n,k-1)$ and move some number of consecutive elements from $\tau_2$ to $\tau_1$ to obtain $\pi=\pi_1n\pi_2$.  This move will ensure that $\pi$ has exactly one more occurrence of $s_1$ than $\tau$ and will also guarantee that $\pi_1$ is not empty.  Note here that $\tau_2$ is guaranteed to be non-empty so this move can always be made.  If $\tau_2$ were to be empty, there would be $n-1$ occurrences of $s_1$, which couldn't possibly be counted by $T(n,k-1)$ as $k<n$.

A more refined structure of a 132-avoiding $n$-permutation $\tau$ is
$$\tau=X_1x_1X_2x_2\ldots X_ix_inY_1y_1Y_2y_2\ldots Y_jy_j$$
where each $X_s$ and $Y_t$ are possibly empty 132-avoiding permutations on their respective elements, $\{x_s\}_{s=1}^i$ is the sequence of right-to-left maxima in $\tau_1$, and $\{y_t\}_{t=1}^j$ is the sequence of right-to-left maxima in $\tau_2$.  Also whenever $s<t$, each element of $X_s$, if any, is larger than every element of $X_t$, if any, and each element of $Y_s$, if any, is larger than every element of $Y_t$, if any.  Again recall that in this case $y_1$ exists, i.e. there is at least one element to the right of $n$.

Now consider the permutation
$$\pi=X_1x_1X_2x_2\ldots X_ix_iY_1y_1nY_2y_2\ldots Y_jy_j$$
obtained from $\tau$ by moving the largest element $y_1$ to the right of $n$, together with the preceding block $Y_1$, to the other side of $n$.  We claim that the 132-avoiding $n$-permutation $\pi$ has exactly one more occurrence of the pattern $s_1$ as desired and that this operation is clearly invertible. Using reasoning described at the beginning of the proof, every element in $Y_1$, if any, was the bottom element of exactly one occurrence of $s_1$ in $\tau$. After the move, it continues to be the bottom element of such an occurrence in $\pi$. On the other hand, $y_1$ was not the bottom element of an occurrence of $s_1$ in $\tau$.  However, after the move, it becomes the bottom element of exactly one occurrence of $s_1$ in $\pi$, namely $y_1n$. Thus, each permutation counted by $T(n,k-1)$ can be transformed using the move described above to a permutation counted by $T(n,k)$.  Because this process is invertible, there are no other permutations counted by $T(n,k)$ apart from those counted by $T(n-1,k)+T(n,k-1)$. \end{proof}

As a byproduct to our research, we define the following set of sequences counted by the Catalan numbers.

\begin{proposition}\label{byproductCatalan1} Let $A_n$ denote the number of sequences $\{a_1,a_2,\ldots,a_n\}$ satisfying $a_1=0$ and $0\leq a_{i}\leq i-1-\displaystyle\sum_{j=1}^{i-1}a_j$ for $i\in\{2,\ldots,n\}$. Then, $A_n$ is given by $C_n$, the $n$-th Catalan number.\end{proposition}

\begin{proof} We will prove that the sequences of length $n$ in question are in one-to-one correspondence with 132-avoiding $n$-permutations, which are known to be counted by $C_n$.

Consider creating all 132-avoiding $n$-permutations by inserting $n$ in every allowable position in every 132-avoiding $(n-1)$-permutation.  It is easily shown, and was already used above, that the only allowable positions to insert $n$ without introducing an occurrence of 132 is either at the beginning, or immediately after a right-to-left maximum. If at every step we label these possible positions to insert $n$ into an $(n-1)$-permutation having $k$ right-to-left maxima from left to right with $0,1,2,\ldots,k$, then we can encode any 132-avoiding permutation as a sequence of choices of where we inserted the current largest element.  Any such sequence must begin with a 0, as the first step always begins with inserting 1 at the beginning of the empty permutation.  For example, the 132-avoiding permutation $\pi=785346291$ is encoded by the sequence $000102013$.

We claim that the set of all sequences $\{a_1,a_2,\ldots a_n\}$ in question are precisely the sequences that encode all 132-avoiding permutations using the method described in the previous paragraph.  This will follow if we can show that for such a sequence the number of right-to-left maxima in the corresponding 132-avoiding $n$-permutation is given by $n-\displaystyle\sum_{j=1}^{n}a_j$.  We prove this statement by induction on $n$.

The case $n=1$ is straightforward as the only sequence is 0 and has corresponding permutation 1, which has exactly $1-0=1$ right-to-left maximum.

Suppose now that $\{a_1,a_2,\ldots,a_n\}$ is a sequence describing insertion of the maximum elements satisfying the conditions specified on the $a_i$'s (that is $0\leq a_{i}\leq i-1-\displaystyle\sum_{j=1}^{i-1}a_j$ for $2\leq i\leq n$), and the number of right-to-left maxima in the corresponding $n$-permutation $\sigma$ is $n-\displaystyle\sum_{j=1}^{n}a_j$. Suppose that $a_{n+1}=k$, where $0\leq k\leq n-\sum_{j=1}^na_j$.  In this situation, we chose to insert $(n+1)$ into $\sigma$ at the position labeled $k$.  It is easy to see that inserting $n+1$ at position $k$ will force each right-to-left maxima preceding this position to become non-right-to-left maxima, but will always create one as $(n+1)$ is a right-to-left maximum.  Thus the number of right-to-left maxima in the obtained permutation will be decreased by $k-1$ when $k\geq1$ or increased by 1 if $k=0$. Thus the number of right-to-left maxima in the resulting permutation is $n-\displaystyle\sum_{j=1}^{n}a_j-(a_{n+1}-1)=n+1-\displaystyle\sum_{j=1}^{n+1}a_j$ and the statement is verified.
\end{proof}

\subsection{The pattern $s_2$ has Catalan's distribution}\label{sub-s2}

Let $M(n,k)$ denote the number of $132$-avoiding $n$-permutations with $k$ occurrences of $s_2$. Clearly, $M(1,0)=1$ and $M(1,k)=0$ for $k>0$. The goal of this section is to prove the following theorem.

\begin{theorem}\label{thm:12TopRightIsCatalanTriangle} We have $M(n,k)=C(n-1,k)$, where $C(n,k)$ is the $(n,k)$-th entry of the Catalan triangle.  Thus, $$M(n,k)=\frac{(n-k){n-1+k\choose n-1}}{n}.$$\end{theorem}

\begin{proof}
We will prove combinatorially that $M(n,k)$ satisfies $M(n,k)=M(n-1,k)+M(n,k-1)$ for $k<n$, a known recursion for the Catalan triangle.

Again, if $\pi=\pi_1n\pi_2$ is an $n$-permutation avoiding the pattern 132, then each element of $\pi_1$, if any, is larger than every element of $\pi_2$, and $\pi_1$ and $\pi_2$ are any 132-avoiding permutations on their respective elements. Moreover, each element of $\pi_1$, if any, is the bottom element of a unique occurrence of the pattern $s_2$, namely the one involving the top element $n$.  It could not be the bottom element of another occurrence of $s_2$ because $n$ would always be present in the shaded area of $s_2$.  The permutations that correspond to $\pi_1$ being empty are counted by the term $M(n,k-1)$.  We will now show that $M(n,k-1)$ is responsible for counting those permutations where $\pi_1$ is not empty.  We accomplish this by providing a general method to take a permutation $\tau=\tau_1n\tau_2$ counted by $M(n,k-1)$ and move some number of consecutive elements from $\tau_2$ to $\tau_1$ to obtain $\pi=\pi_1n\pi_2$.  This move will ensure that $\pi$ has exactly one more occurrence of $s_2$ than $\tau$ and will also guarantee that $\pi_1$ is not empty.  Note here that $\tau_2$ is guaranteed to be non-empty so this move can always be made.  If $\tau_2$ were to be empty, there would be $n-1$ occurrences of $s_2$, which could not possibly be counted by $M(n,k-1)$ as $k<n$.

A more refined structure of a 132-avoiding $n$-permutation $\tau$ is
$$\tau=X_1x_1X_2x_2\ldots X_ix_inY_1y_1Y_2y_2\ldots Y_jy_j$$
where each $X_s$ and $Y_t$ are possibly empty 132-avoiding permutations on their respective elements, $\{x_s\}_{s=1}^i$ is the sequence of right-to-left maxima in $\tau_1$, and $\{y_t\}_{t=1}^j$ is the sequence of right-to-left maxima in $\tau_2$.  Also whenever $s<t$, each element of $X_s$, if any, is larger than every element of $X_t$, if any, and each element of $Y_s$, if any, is larger than every element of $Y_t$, if any.  Again recall that in this case $y_1$ exists, i.e. there is at least one element to the right of $n$.

Now consider the permutation
$$\pi=X_1x_1X_2x_2\ldots X_ix_iY_1y_1nY_2y_2\ldots Y_jy_j$$
obtained from $\tau$ by moving the largest element $y_1$ to the right of $n$, together with the preceding block $Y_1$, to the other side of $n$.  We claim that the 132-avoiding $n$-permutation $\pi$ has exactly one more occurrence of the pattern $s_2$ as desired and that this operation is clearly invertible. Using reasoning described at the beginning of the proof, every element in $Y_1$, if any, was the bottom element of exactly one occurrence of the pattern $s_2$ in $\tau$. After the move, it continues to be the bottom element of such an occurrence in $\pi$, however the top element of the occurrence changes from $y_1$ to $n$. On the other hand, $y_1$ was not the bottom element of an occurrence of $s_2$ in $\tau$.  However, after the move, it becomes the bottom element of exactly one occurrence of $s_2$, namely $y_1n$. Thus, each permutation counted by $M(n,k-1)$ can be transformed using the move described above to a permutation counted by $M(n,k)$.  Because this process is invertible, there are no other permutations counted by $M(n,k)$ apart from those counted by $M(n-1,k)+M(n,k-1)$.
\end{proof}

Thus, we obtain a new combinatorial interpretation of $C(n-1,k)$, namely the number of 132-avoiding $n$-permutations having exactly $k$ occurrences of $s_2$.  Using this new interpretation, we were able to provide a relation which seems to be new on the Catalan triangle, and we record it as the following theorem.

\begin{theorem}\label{thm:CatalanTriangleRelation} For the Catalan triangle,
$$C(n,k)=\sum_{i=0}^{k}C_iC(n-i-1,k-i).$$
\end{theorem}
\begin{proof}
As reasoned in the proof of Theorem \ref{thm:12TopRightIsCatalanTriangle}, if $\pi=\pi_1n\pi_2$ is an $n$-permutation avoiding the pattern 132, then each element of $\pi_1$, if any, is larger than every element of $\pi_2$, and $\pi_1$ and $\pi_2$ are any 132-avoiding permutations on their respective elements. Moreover, each element of $\pi_1$, if any, is the bottom element of a unique occurrence of the pattern $s_2$, namely the one involving the top element $n$.

Additionally, $\pi_1n$ has no affect on occurrences of $s_2$ inside $\pi_2$. Since there are $C_n$ 132-avoiding $n$-permutations, we get the following recursion for $M(n,k)$:
\begin{equation}\label{eq:CatalanTriangleIdentity}M(n,k)=\sum_{i=0}^{k}C_iM(n-i-1,k-i)\end{equation} with initial conditions $M(n,0)=1$ and $M(n,n)=0$ for all $n\geq 1$.  Theorem \ref{thm:12TopRightIsCatalanTriangle} verifies that equation (\ref{eq:CatalanTriangleIdentity}) is equivalent to the statement of the theorem.
\end{proof}

Using known formulas for both $C_i$ and $C(n,k)$, Theorem \ref{thm:CatalanTriangleRelation} gives rise to the following binomial identity.
\begin{corollary} \label{binom} We have
$$\frac{(n-k){n-1+k\choose n-1}}{n}=\sum_{i=0}^{k}\frac{{2i\choose i}}{i+1}\frac{(n-k-1){n-2i-2+k\choose n-i-2}}{n-i-1}.$$
\end{corollary}

\begin{remark} We have used the same recursion for the Catalan triangle in the proof of Theorem {\rm\ref{thm:12TopRightIsCatalanTriangle}} and in the second proof of Theorem {\rm\ref{thm:Boxed12IsCatalanTriangle}}, which induces a bijective proof of the equidistribution of the mesh patterns $s_1$ and $s_2$ on $132$-avoiding permutations. \end{remark}

\subsection{The patten $s_3$ has the reverse Catalan distribution}\label{sub-s3}

We let $N(n,k)$ denote the number of $132$-avoiding $n$-permutations with $k$ occurrences of the pattern $s_3$. Clearly, $N(1,0)=1$ and $N(1,k)=0$ for $k>0$.

\begin{theorem}\label{thm:ReverseDiagonal12IsReverseCatalanTriangle} $N(n,k)=C(n-1,n-1-k)$, where $C(n,k)$ is the $(n,k)$-th entry of the Catalan triangle.  Thus, $$N(n,k)=\frac{(k+1){2(n-1)-k\choose n-1}}{n}.$$\end{theorem}

\begin{proof}
We explain this fact by establishing a bijection $f$ which we will show is
actually an involution between 132-avoiding $n$-permutations having $k$ occurrences ($0\leq
k\leq n-1$) of the pattern $s_3$ and 132-avoiding $n$-permutations having
$n-1-k$ occurrences of the pattern $s_1$ (of which there are $C(n-1,n-1-k)$ many by Theorem \ref{thm:Boxed12IsCatalanTriangle}).

It is easy to see that the bottom elements of occurrences of $s_3$ in a 132-avoiding permutation $\pi$
are precisely those elements that are both a left-to-right minimum and a
right-to-left maximum in the permutation obtained from $\pi$ by removing all right-to-left
maxima. At the same time, the only elements in $\pi$ that are not the bottom elements of an occurrence
of $s_1$ are precisely the right-to-left maxima of $\pi$.

We will provide a map that will turn
the position of a bottom element of an occurrence of $s_3$, as well as the rightmost
element of an $n$-permutation, into the position of a right-to-left maximum (which as mentioned before
can not be the bottom element of an occurrence of $s_1$) and vice-versa.

Let $\pi=\pi_1\pi_2\cdots\pi_n$ be a $132$-avoiding permutation with
right-to-left maxima in positions $1\leq i_1<i_2<\cdots<i_m=n$ and let $I_\pi=\{i_1,\ldots,i_m\}$.
Note that $\pi_{i_1}$ must be $n$.  Also, let the bottom elements of
occurrences of $s_3$ in $\pi$ be in positions $1\leq
a_1<a_2<\cdots<a_{s-1}<n$, $a_s:=n$ and let $A_\pi=\{a_1,\ldots,a_s\}$. Note that the sets
$I_\pi$ and $A_\pi$ are both guaranteed to contain $n$ and this is the only element they share in common.

We will first build an auxiliary permutation $\pi'=\pi'_1\pi'_2\cdots\pi'_n$ by cyclically shifting
the elements having positions in $I_\pi$ to the left.  Formally,
\begin{itemize}
\item $\pi'_i:=\pi_i\mbox{ if }i\not\in I_\pi$\\
\item $\pi'_{i_j}:=\pi_{i_{j+1}}\mbox{ for }j=1,2,\ldots,m-1\mbox{ and }$\\
\item $\pi'_{i_m}:=\pi_{i_1}=n$.
\end{itemize}

We now build the image of $\pi$,
$\sigma=\sigma_1\sigma_2\cdots\sigma_n=f(\pi)$ from $\pi'$ by cyclically shifting
the elements having positions in $A_\pi$ to the right.  Formally,
\begin{itemize}
\item $\sigma_i:=\pi'_i\mbox{ if }i\not\in A_\pi$\\
\item $\sigma_{a_i}:=\pi'_{a_{i-1}}\mbox{ for }2\leq i\leq s\mbox{ and }$\\
\item $\sigma_{a_1}=\pi'_{a_s}$.
\end{itemize}

\begin{figure}[ht]
\begin{center}
\includegraphics[scale=0.5]{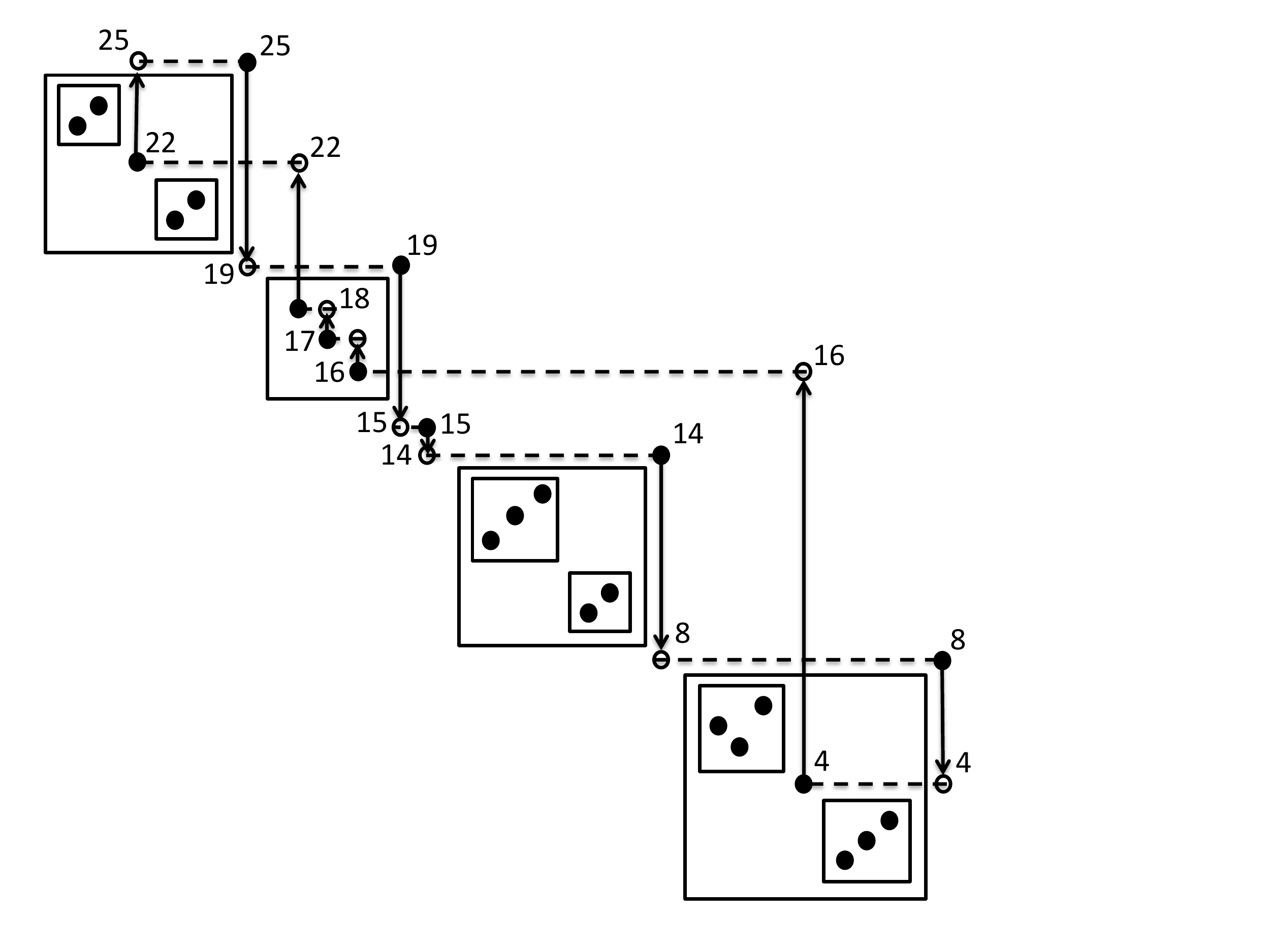}
\end{center}
\vspace{-15pt}
\caption{(23)(24)(22)(20)(21)(25)(18)(17)(16)(19)(15)(11)(12)(13)9(10)(14)65741238 and its
image under the involution $f$.}\label{exampleInvolution}
\end{figure}

To demonstrate the map, we will show how it acts on the 132-avoiding permutation
$$\pi=(23)(24)(22)(20)(21)(25)(18)(17)(16)(19)(15)(11)(12)(13)9(10)(14)65741238.$$
It is easy to verify that for $\pi$, $I_\pi=\{6,10,11,17,25\}$ and $A_\pi=\{3,7,8,9,21,25\}$.
We first perform the cyclic shift among the elements having positions in $I_\pi$ to the left to obtain
$$\pi'=(23)(24)(22)(20)(21)(19)(18)(17)(16)(15)(14)(11)(12)(13)9(10)86574123(25).$$  Now
we perform the cyclic shift among the elements having positions in $A_\pi$ to the right to obtain the image of the map
$$\sigma=(23)(24)(25)(20)(21)(19)(22)(18)(17)(15)(14)(11)(12)(13)9(10)8657(16)1234.$$  Note that while in the image permutation $\sigma$,
\[
I_\sigma=\{3,7,8,9,21,25\}\text{ and }A_\sigma=\{6,10,11,17,25\}
\]
so that $I_{f(\pi)}=A_\pi$ and $A_{f(\pi)}=I_\pi$.
In fact, we will show that this is always the case and with this fact, $f$ is easily seen to be an involution.

The permutation matrix for $\pi$ and $f(\pi)$ are shown in Figure \ref{exampleInvolution}.
In this figure, the elements in smaller boxes without arrows are unchanged in value/position, while the other solid circles belong to the original permutation and the open circles belong to its image; for each unfixed position, there is a vertical arrow from the value of $\pi$ to the value of $f(\pi)$ and the horizontal lines were included to show the levels on which elements are moved.

To check understanding, the reader can verify that image of $$\tau= (15)(14)(12)(13)(16)89(10)(11)5643127$$ under $f$ is $$(16)(15)(12)(13)(11)89(10)756(14)4123.$$

It was mentioned earlier that $I_{f(\pi)}=A_\pi$ and $A_{f(\pi)}=I_\pi$.  This is a direct consequence of the definition of the map and is most easily understood visually with help of Figure~\ref{exampleInvolution}.  However, for those who appreciate technical detail, we shall provide rigorous arguments in the following two paragraphs, for those who don't, feel free to skip over them.

Suppose we have a 132-avoiding permutation $\pi$ having corresponding sets $A_\pi$ and $I_\pi$ with $i_j\in I_\pi$ and $i_j\neq n$.  There is no need to consider the case when $i_j=n$ as $n\in A_{f(\pi)}$ by definition. Let us examine what happens at position $i_j$ throughout the map $f$.  After the first cyclic shift, we have that $\pi'_{i_j}$ has only one element larger than it located to its right, namely $n$, and has no elements smaller than it located to its left.  The latter fact is simply because if a smaller element did exist to its left then it would have been the bottom element of a 132 pattern in $\pi$ along with the elements at positions $i_j$ and $i_{j+1}$ (which exists since $i_j\neq n$).  Consider the first position, $a_k$, to the right of $i_j$ that is in $A_\pi$.  After the second cyclic shift, we claim that the value at position $a_k$ in $f(\pi)$ is greater than $\pi'_{i_j}$ and is thus a right-to-left maximum, ensuring that there is an occurrence of $s_3$ having bottom element at position $i_j$ and top element at position $a_k$.  To see this, realize that if there is no element of $A_\pi$ less than $i_j$, the claim is immediate as the element at position $a_k$ will be $n$ in this situation.  Otherwise, we know that the value of $\pi$ at position $a_{k-1}$ (which is precisely the value of the $f(\pi)$ at position $a_k$) is larger than $\pi'_{i_j}$ because if not, then the elements at positions $a_{k-1}$, $i_j$ and $i_{j+1}$ in $\pi$ would have been an occurrence of 132.  Also, in this case, the element $n$ must now precede $f(\pi)_{a_k}$, which ensures that it is also a right-to-left maximum.  This verifies that $i_j$ belongs to the set $A_{f(\pi)}$.

Now suppose we have a 132-avoiding permutation $\pi$ having corresponding sets $A_\pi=\{a_1,\ldots,a_s\}$ and $I_\pi$ with $a_k\in A_\pi$ and $a_k\neq n$.  Again, we will examine what happens at position $a_k$ throughout the map $f$.  To begin, there is only one element of $\pi$ to the right of position $a_k$ that is larger than $\pi_{a_k}$ and this element $i_j\in I_\pi$ must be a right-to-left maximum of $\pi$.  After the first cyclic shift, $\pi'_{i_j}$ is either smaller than $\pi'_{a_k}$ or $\pi'_{i_j}=n$.  If $\pi'_{i_j}\neq n$ and was larger than $\pi'_{a_k}$, then the elements at positions $a_k$, $i_j$ and $i_{j+1}$ in $\pi$ would have been an occurrence of 132.  In either case, we have that $\pi'_n=n$ is the only element following position $a_k$ that is larger than $\pi'_{a_k}$.  Additionally, we have that $\pi'_{a_k}<\pi'_{a_{k-1}}$ for $2\leq k\leq s$ and that $\pi_{a_1}<\pi'_{a_s}=n$.  These facts guarantee that after the second cyclic shift, the element of $f(\pi)$ at position $a_k$ will be a right-to-left maximum as $n$ will now be forced to either precede $f(\pi)_{a_k}$ or actually be equal to the value of $f(\pi)$ at this position.  This verifies that $a_k$ belongs to the set $I_{f(\pi)}$.

We will now show that $f$ is actually an involution by showing that the image of $f$ is actually a 132-avoiding permutation. To this end, for any permutation $\pi$, note that we do not change values of elements whose positions do not belong to $I_\pi \cup A_\pi=I_{f(\pi)}\cup A_{f(\pi)}$.  Thus, if an occurrence of the pattern $132$ did exist in $f(\pi)$, it must involve three elements, one of which must be in $I_\pi \cup A_\pi$. Suppose $i\in I_\pi \cup A_\pi$. If $i\in I_\pi$, then $i\in A_{f(\pi)}$ and therefore only one element to the right of position $i$ can be larger than $f(\pi)_i$.  If $i\in A_\pi$, then $i\in I_{f(\pi)}$ and therefore no element to the right of position $i$ can be larger than $f(\pi)_i$.  These facts guarantee that the element at position $i$ can not play the role of 1 in a 132 pattern.  Similarly, one can show that $i$ can not play the role of 3 in an occurrence of the pattern 132, since either there are no elements smaller than $\pi_i$ to the left of it, or everything to the left of $\pi_i$ is larger than everything to the right of it.  Finally, $i\neq n$ can not play the role of 2 in an occurrence of 132 because again, since either there are no elements smaller than $\pi_i$ to the left of it, or there are no elements smaller than $\pi_i$ to the left of the nearest right-to-left maximum preceding $\pi_i$.
\end{proof}

\begin{remark} The fact that the map $f$ in the proof of Theorem {\rm \ref{thm:ReverseDiagonal12IsReverseCatalanTriangle}} is actually an involution proves a more general fact, namely that the pairs of statistics  $(s_3,n-1-s_1)$ and $(n-1-s_1,s_3)$ are equidistributed on $132$-avoiding permutations. In fact,  the involution $f$ actually gives a more general fact on joint equidistribution of quadruples of statistics $(s_3,v(s_3),n-1-s_1,v(non s_1))$
and $(n-1-s_1,v(non s_1),s_3,v(s_3))$ where $v(p_3)$ is a binary vector
showing positions in which the bottom elements of occurrences of $s_3$
occur; e.g., for $s_3$ $0010011$ would mean that the bottom elements of
$s_3$ are in the $3$rd and $6$th positions, and we can assume that the
rightmost element is always $1$ by definition. The meaning
of $v(non s_1)$ should be clear: this is a binary vector recording positions of elements which are not bottom elements in occurrences of $s_1$. \end{remark}

\subsection{On the pattern $t_1$}\label{sub-t1}

In this subsection, we discuss the minimum and the maximum number of occurrences of the pattern $t_1$ on 132-avoiding permutations. Suppose $\pi_i<\pi_j$ for $i<j$ in a permutation $\pi=\pi_1\pi_2\cdots\pi_n$. By the {\em box between $\pi_i$ and $\pi_j$} we mean the rectangle in the permutation matrix of $\pi$ defined by $i$, $j$, $\pi_i$ and $\pi_j$.

It is easy to see, and first was observed in \cite[Theorem 4]{AKV}, that on 132-avoiding permutations, avoidance of $t_1$ is equivalent to avoidance of the classical pattern 123. Indeed, if we avoid 123 we clearly avoid $t_1$. On the other hand, if we contain 123, then there must exist an occurrence of 123 which is also an occurrence of  $t_1$. Indeed, for any occurrence $xyz$ of 123, if it is not an occurrence of $t_1$, select the smallest element greater than $y$ positioned strictly to the right of $y$ but to the left of $z$ (possibly including $z$), and the largest element less than $y$ positioned strictly to the left of $y$ but to the right of $x$ (possibly including $x$).  These two chosen elements together with $y$ will form an occurrence of $t_1$. Thus, simultaneous avoidance of 132 and $t_1$ is equivalent to simultaneous avoidance of 132 and 123, which is known to give cardinalities $2^{n-1}$ for $n$-permutations.

In what follows, we will explain our observation that the number of 132-avoiding permutations with the maximum number of occurrences of $t_1$ is given by the Catalan numbers.

\begin{lemma}\label{132lemma1} In a $132$-avoiding $n$-permutation, each element is the bottom element of at most one occurrence or $t_1$ and thus taking into account that $n$ and $n-1$ cannot be bottom elements of such occurrences, the maximum number of occurrences of $t_1$ is no more than $n-2$ (which is attainable by, e.g. $12\cdots n$).\end{lemma}

\begin{proof} Let $xyz$ be an occurrence of $t_1$, where $x$ is the bottom element. Suppose $xy'z'$ is another occurrence of $t_1$. If $y'$ is to the left of $y$ then $y'$ must be larger than $y$ (because $xyz$ is an occurrence of $t_1$ and $y'$ must be outside of the box between $x$ and $y$), but then $xy'y$ is an occurrence of the pattern $132$, which is a contradiction. On the other hand, if $y'$ is to the right of $y$, then either $y$ will be in the prohibited area between $x$ and $y'$ (if $y'>y$), or $xyy'$ is an occurrence of the pattern $132$ which is also a contradiction. Thus $y=y'$.

Similarly, if $z'$ is to the left of $z$, then $z'$ must be larger than $z$ (not to be in the box between $y$ and $z$), and we get that $yz'z$ is an occurrence of the pattern $132$. On the other hand, if $z'$ is to the right of $z$, then either $z$ is in the prohibited area between $z'$ and $y$, or $yzz'$ is an occurrence of the pattern $132$.  Thus $z=z'$, and $xyz$ is the only occurrence of $t_1$ having $x$ as bottom element.\end{proof}

\begin{lemma}\label{132lemma2} In a $132$-avoiding $n$-permutation with maximum number of occurrences of $t_1$, $n$ must be the rightmost element and it must be preceded by $n-1$. \end{lemma}

\begin{proof} Clearly, a right-to-left maximum cannot be the bottom element of an occurrence of $t_1$, and thus, by Lemma \ref{132lemma1}, an optimal permutation has at most two right-to-left maxima.

Any 132-avoiding permutation $\pi$ has the structure $\pi_1n\pi_2$, where each element of $\pi_1$, if any, is larger than every element of $\pi_2$, if any. Assume $\pi$ has two right-to-left maxima, which makes $\pi_2$ non-empty.  One of the two must be $n$ and the other we will call $x$ (the rightmost element of $\pi$). We see that the maximum element to the left of $n$ and the maximum element between $n$ and $x$ (at least one of them exists assuming the length of $\pi$ is at most 3), cannot be the bottom element of an occurrence of $t_1$. Thus, $\pi$ has at most $n-3$ occurrences of $t_1$ and cannot be optimal.  This tells us that $n$ must be the rightmost element (the single right-to-left maximum). Therefore, we can assume $\pi=\pi'n$. However, a right-to-left maximum in $\pi'$ cannot be the bottom element of an occurrence of $t_1$, and thus, having more than one right-to-left maxima in $\pi'$ would mean having at most $n-3$ occurrences of $t_1$ in $\pi$. Thus, we must have $\pi=\pi''(n-1)n$ for some permutation $\pi''$ \end{proof}

\begin{lemma}\label{132lemma3} Suppose $\pi=\pi''(n-1)n$ is a $132$-avoiding permutation. Then the number of occurrences of $t_1$ in $\pi$ is $(n-2)$, and thus, by Lemma {\rm\ref{132lemma1}}, $\pi$ has the maximum possible number of occurrences of $t_1$. \end{lemma}

\begin{proof} Let $x$ be an element of $\pi''$. Our goal is to show that $x$ is the bottom element of an occurrence of $t_1$ in $\pi$. Indeed, $x(n-1)n$ is an occurrence of the mesh pattern $r_1$ shown in Figure \ref{aux1}.
If there are no elements in the box defined by $x$ and $(n-1)$, we are done. Otherwise, let $y$ be the smallest element in that box. Clearly $xy(n-1)$ is an occurrence of the mesh pattern $r_2$ shown in Figure \ref{aux1}.
If this is actually an occurrence of $t_1$, we are done. Otherwise, let $z$ be the minimum element in the box defined by $y$ and $(n-1)$. Then $xyz$ is an occurrence of $t_1$.\end{proof}

\begin{figure}[ht]
\begin{center}
\includegraphics[scale=0.4]{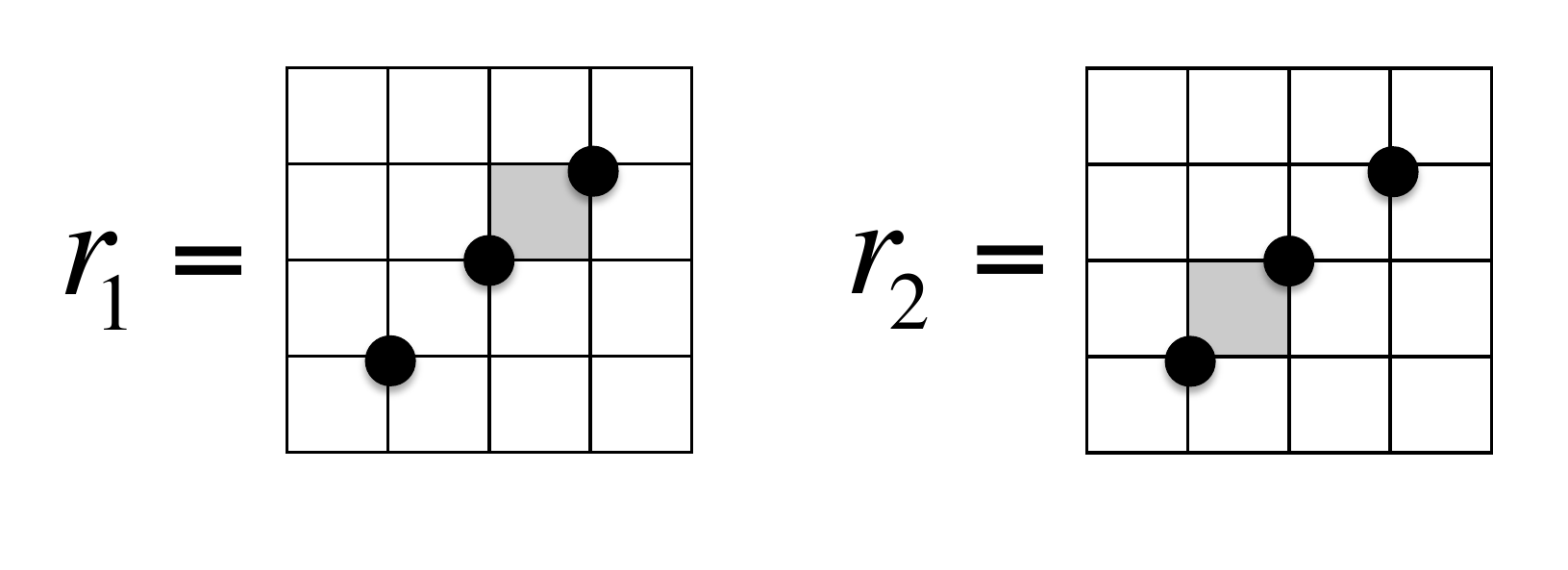}
\end{center}
\vspace{-20pt}
\caption{Two auxiliary mesh patterns.}\label{aux1}
\end{figure}

\begin{theorem} \label{maxt1} There are $C_{n-2}$ $132$-avoiding $n$-permutations that contain the maximum number of occurrences of $t_1$.\end{theorem}

\begin{proof} The result follows from Lemmas \ref{132lemma2} and \ref{132lemma3} since the permutation $\pi''$ in Lemma \ref{132lemma3} is an arbitrary $132$-avoiding permutation known to be counted by $C_{n-2}$.\end{proof}

\subsection{On the pattern $t_2$}\label{sub-t2}

In this subsection, we will prove a number of observations regarding the distribution of $t_2$ on $132$-avoiding permutations.

\begin{lemma}\label{q132lemma1} If a $132$-avoiding $n$-permutation $\pi$ contains at least one occurrence of $t_2$, then it ends with $n$ and contains an occurrence of $t_2$ that involves $n$. \end{lemma}

\begin{proof}  Suppose that $n$ is not at the end of $\pi$ and $\pi=\pi_1n\pi_2$ with $\pi_2$ non-empty. Clearly, because of an element in $\pi_2$, $t_2$ cannot occur in $\pi_1n$, and because of $n$, $t_2$ cannot occur in $\pi_2$.  No other occurrence of $t_2$ could exist in $\pi$ because each element in $\pi_1n$ is larger than every element in $\pi_2$, which is a contradiction.  Thus $\pi$ must end with $n$.

Now, suppose that $xy$ is an occurrence of $t_2$ in a permutation $\pi$ and $y\neq n$.  Assume that there is at least one element in the box defined by $y$ and $n$.  If not, $yn$ is the desired occurrence of $t_2$. Otherwise, one can take the topmost element in that box together with $n$ to get the desired occurrence of $t_2$.  This completes the proof. \end{proof}

\begin{remark} Note that in the box defined by $y$ and $n$ in Lemma {\em \ref{q132lemma1}}, all elements, if any, must be in increasing order to avoid the pattern $132$.\end{remark}

\begin{theorem}\label{q132proposition1}  For $n\geq 3$, there are no $132$-avoiding permutations with exactly one occurrence of the pattern $t_2$.\end{theorem}

\begin{proof}  Suppose there exists a $132$-avoiding permutation $\pi$ of length $n$ that has exactly one occurrence of $t_2$. By Lemma \ref{q132lemma1}, $\pi$ ends with $n$ and the only occurrence of $t_2$ in $\pi$ involves $n$.  Suppose that $\pi=\pi_1n$ and the single occurrence of $t_2$ is $xn$ for some $x$ in $\pi_1$. In the permutation matrix of $\pi$, the element $x$ subdivides $\pi_1$ into four quadrants (using usual way to label quadrants in counterclockwise direction), where quadrant I is empty.  If quadrant II is not empty, then its largest element together with $n$ would be a different occurrence of $t_2$. Similarly, if quadrant IV is not empty, then its largest element together with $n$ would also be a different occurrence of $t_2$. Lastly, if quadrant III is not empty, then its topmost element together with $x$ would be another occurrence of $t_2$. Thus all four quadrants are empty, which contradicts the assumption that $n\geq 3$. Therefore, no such $\pi$ can exist.\end{proof}

As a direct corollary to the proof of Theorem \ref{q132proposition1}, we have the following theorem.

\begin{theorem}\label{q132proposition2}  The number of $132$-avoiding $n$-permutations that avoid $t_2$ is given by $C_n-C_{n-1}$. \end{theorem}

\begin{proof} The number of $n$-permutations avoiding the pattern $132$ is $C_n$. Again, by Lemma \ref{q132lemma1}, if such a permutation contains $t_2$ then $n$ must be the rightmost element, and therefore could not be part of a $132$ pattern in the permutation.  Thus, the number of 132-avoiding $n$-permutations with $n$ at the end is given by $C_{n-1}$.  \end{proof}

A {\em Dyck path} of length $2n$ is a lattice path from $(0,0)$ to $(2n,0)$ with steps $U=(1,1)$ and $D=(1,-1)$ that never goes below the $x$-axis. In the {\em standard bijection} between $132$-avoiding permutations and Dyck paths, the position of the largest element corresponds to the leftmost return of a path to the $x$-axis (see \cite{kit} for details). Using Theorem \ref{q132proposition2}, we can easily map bijectively such permutations to, e.g., Dyck paths that do not start with $UU$. Indeed, Dyck paths that begin with $UD$ can be mapped to 132-avoiding permutations ending with the largest element through the standard bijection composed with applying reverse to Dyck paths. The same map will then send (132,$t_2$)-avoiding permutations to Dyck paths beginning with $UU$.

The following lemma is straightforward to prove.

\begin{lemma}\label{q132lemma2}  Let $\pi=\pi'n$ be a $132$-avoiding $n$-permutation. We have that $x$ is a right-to-left maximum in $\pi'$ if and only if $xn$ is an occurrence of $t_2$. \end{lemma}

\begin{theorem}\label{q132proposition3} For $n\geq 4$, there are $C_{n-2}$ $132$-avoiding $n$-permutations that contain exactly two occurrences of $t_2$. There are two such permutations of length $3$ and none of smaller lengths.\end{theorem}

\begin{proof}  To contain occurrences of $t_2$, Lemma \ref{q132lemma1} forces the structure of $n$-permutations in question to be $\pi'n$. By Lemma \ref{q132lemma2}, $\pi'$ has either one or two right-to-left maxima.

If $\pi'$ has one right-to-left maximum, then $\pi=\pi''(n-1)n$. We see that $(n-1)n$ is an occurrence of $t_2$, and unless $\pi=123$, Theorem \ref{q132proposition1} verifies that it is not possible to have exactly one occurrence of $t_2$ in $\pi''$ which is equivalent to having exactly two occurrences of $t_2$ in $\pi$.

Suppose that $\pi'$ has two right-to-left maxima, $n-1$ and $x$.  Note that $x$ must be the rightmost element of $\pi'$. Clearly, $(n-1)n$ and $xn$ are both occurrences of $t_2$. Also there are no elements to the right of $n-1$ that are larger than $x$.  On the other hand since $\pi$ must avoid the pattern 132, no element to the left of $n-1$ is smaller than $x$.  Further, because of $x$, no element to the left of $(n-1)$ (all of which are larger than $x$) can be the bottom element of an occurrence of $t_2$, and such elements can form any $132$-avoiding permutation.  Similarly, no element to the right of $(n-1)$ that is smaller than $x$ can be the bottom element of an occurrence of $t_2$ because of $n-1$ and these elements can form any $132$-avoiding permutation. Thus, assuming $A_{2,n}$ denotes the number of $132$-avoiding $n$-permutations with exactly two occurrences of $t_2$, we have, for $n\geq 4$, the following recursion, where $i$ stands for the number of elements to the left of $(n-1)$:
$$A_{2,n}=\sum_{i=0}^{n-3}C_iC_{n-3-i}=C_{n-2}.$$
The case $n=3$ is $A_{2,3}=2$, which is given by the permutations 123 and 213. \end{proof}

\begin{theorem}\label{q132proposition4} For $n\geq 5$, there are $C_{n-2}$ $132$-avoiding $n$-permutations that contain exactly three occurrences of $t_2$. There are three such permutations of length $4$ and none of smaller lengths.\end{theorem}

\begin{proof} Let $n\geq 4$ (clearly for smaller lengths we have no ``good" permutations). By Lemma \ref{q132lemma1}, to contain occurrences of $t_2$, the structure of an $n$-permutation $\pi$ in question must be  $\pi'n$. By Lemma \ref{q132lemma2}, $\pi'$ has either one or two or three right-to-left maxima.

If $\pi'$ has one right-to-left maximum, then either $\pi$ ends with $(n-2)(n-1)n$ or it ends with $(n-1)n$ and there are two right-to-left maxima in the permutation obtained from $\pi$ by removing $(n-1)n$. In the former case, unless $\pi=1234$, by Theorem \ref{q132proposition1}, we have no ``good" permutations, while in the later case we can apply Theorem \ref{q132proposition3}  to obtain $C_{n-3}$ permutations containing exactly three occurrence of $t_2$ (note, that $(n-1)n$ is an occurrence of $t_2$).

If $\pi'$ has two right-to-left maxima, then they, together with $n$, will form two occurrences of $t_2$, and following the arguments in Theorem \ref{q132proposition3} we will see that there are no other occurrences of $t_2$. Thus, there are no ``good" permutations in this case.

Finally, if $\pi'$ has three right-to-left maxima, say $(n-1)>x>y$, then we can argue in a similar way as in the proof of Theorem \ref{q132proposition3} to see that $\pi=\pi'_1(n-1)\pi'_2x\pi'_3yn$ where each element of $\pi'_1$, if any, is larger than any element in $\pi'_2$ and $\pi'_3$, and each element of $\pi'_2$, if any, is larger than any element in $\pi'_3$. Moreover, $\pi'_1$ (resp., $\pi'_2$ and $\pi'_3$) is any $132$-avoiding permutation not contributing to extra occurrences of $t_2$. Thus, if $n\geq 5$ and $A_{3,n}$ is  the number of $132$-avoiding $n$-permutations with exactly three occurrences of $t_2$,  we have the following recursion (here $i$ is the number of elements to the left of $(n-1)$ and $j$ is the number of elements between $(n-1)$ and $x$):
$$A_{3,n}=C_{n-3}+\sum_{i=0}^{n-4}C_i\sum_{j=0}^{n-4-i}C_jC_{n-4-i-j}=C_{n-3}+\sum_{i=0}^{n-4}C_iC_{n-3-i}=\sum_{i=0}^{n-3}C_iC_{n-3-i}=C_{n-2}.$$
The case $n=4$ is $A_{3,4}=3$, which is given by the permutations 1234, 2134 and 3214.
\end{proof}

\begin{remark}\label{remarkLast} Note that by Theorems {\em \ref{q132proposition3}} and {\em \ref{q132proposition4}}, for $n\geq 4$, the numbers of $132$-avoiding $n$-permutations containing exactly two and exactly three occurrences of $t_2$ coincide. A natural question here is to provide a combinatorial proof of this fact.\end{remark}

\section{Concluding remarks}\label{concluding}

Studying the distribution of occurrences of a given pattern on various sets of permutations is typically a very hard problem, but in this paper we were able to solve this problem explicitly for four patterns, namely $p$, $s_1$, $s_2$ and $s_3$, providing links to two well-known objects -- the harmonic numbers and Catalan's triangle. A natural research direction here is to continue this study for other patterns.  If determining the complete distribution seems difficult, one could attempt to provide formulas for particular cases like we have done for the patterns $t_1$ and $t_2$ on 132-avoiding permutations.  Also, the patterns we have studied on 132-avoiding permutations ($s_1$, $s_2$, $s_3$, $t_1$ and $t_2$) can be studied over all permutations or other sets of restricted permutations. Likely, such studies will bring bijective questions, like the one mentioned in Remark \ref{remarkLast}.

\section{Acknowledgments}

We are grateful to Sofia Archuleta for useful discussions, in particular helping discover and verify the bijection presented in Section \ref{sec-q1-q2}.

\end{document}